\documentclass[11pt]{amsart}
 \usepackage{amssymb}
\usepackage{url}
\usepackage{longtable}
\usepackage{color}

\definecolor{dgr}{rgb}{0.1,0.7,0.2} 

\numberwithin{equation}{section}

\newtheorem{lem}{Lemma}[section]
\newtheorem{thrm}[lem]{Theorem}
\newtheorem{prob}[lem]{Problem}
\newtheorem{cor}[lem]{Corollary}

\newtheorem{defin}[lem]{Definition}

\DeclareMathOperator\EP{\mathbf{EP}}
\DeclareMathOperator\AP{\mathbf{AP}}
\DeclareMathOperator\NEP{\mathbf{NEP}}
\DeclareMathOperator\NAP{\mathbf{NAP}}

\def\Z{{\mathbb Z}}
\def\C{{\rm Z}}

\def\la{{\langle}}
\def\ra{{\rangle}}
\def\cal{{\mathcal}}

\def\gcd{{\rm gcd}}
\def\tl{\mathbin{\triangleleft}}

\def\AGL{{\rm AGL}}
\def\AGAMMAL{{\rm A\Gamma L}}

\def\mod{{\rm mod\ }}
\def\cal{\mathcal}

\def\gen#1{\left\langle #1\right\rangle}

\newcommand{\PGammaU}{{\rm P \Gamma U}}
\newcommand{\PSL}{{\rm PSL}}
\newcommand{\PSp}{{\rm PSp}}
\newcommand{\SL}{{\rm SL}}
\newcommand{\PSU}{{\rm PSU}}
\newcommand{\SU}{{\rm SU}}
\newcommand{\SD}{{\rm SD}}
\newcommand{\PGL}{{\rm PGL}}
\newcommand{\ASL}{{\rm ASL}}
\newcommand{\GL}{{\rm GL}}

\newcommand{\Mat}{{\rm M}}

\newcommand{\PGAMMAL}{{\rm P\Gamma L}}
\newcommand{\SigmaU}{{\rm \Sigma U}}
\newcommand{\PSigmaL}{{\rm P\Sigma L}}

 \newcounter{case}

 \renewcommand{\thecase}{\arabic{case}}

\newcounter{subcase}

 \renewcommand{\thesubcase}{\alph{subcase}}

\begin{document}

\title[New Hierarchy]{Imprimitive Permutations in Primitive Groups}

\author{J. Ara\'ujo}
\address{Universidade Aberta and CEMAT-Ci\^{e}ncias
Faculdade de Ci\^{e}ncias, Universidade de Lisboa,
 1749-016, Lisboa, Portugal}
 \email{jjaraujo@fc.ul.pt}
\author{J. P. Ara\'ujo}
\address{Instituto Superior T\'{e}cnico,
Universidade de Lisboa,
 1749-016, Lisboa, Portugal
}
\email{joao.p.araujo@tecnico.ulisboa.pt}
\author{P. J. cameron}
\address{Mathematical Institute\\ University of St Andrews\\
St Andrews, Fife KY16 9SS\\Scotland}
\email{\url{pjc20@st-andrews.ac.uk }}
\author{T. Dobson}
\address{Department of Mathematics and Statistics, Mississippi State University, PO Drawer MA Mississippi State, MS 39762 USA,
and
IAM,
University of Primorska,
Koper 6000, Slovenia}
\email{dobson@math.msstate.edu}
\author{A. Hulpke}
\address{Department of Mathematics,
Colorado State University,
1874 Campus Delivery,
Fort Collins, CO 80523-1874, USA}
\email{hulpke@math.colostate.edu}
\author{P. Lopes}
\address{Department of Mathematics and CAMGSD, Instituto Superior T\'{e}cnico,
Universidade de Lisboa,
 1049-001, Lisboa, Portugal}
\email{pelopes@math.tecnico.ulisboa.pt}
\date{}
\begin{abstract}
{\color{black} The goal of this paper is to study  primitive groups  that are contained in the union of maximal (in the symmetric group) imprimitive groups. The study of types of permutations inside primitive groups goes back to the origins of the theory of permutation groups. However, this is another instance of a  situation common in mathematics in which a very natural problem turns out to be extremely difficult. Fortunately, the enormous progresses of the last few decades seem to allow a new momentum on the attack to this problem.  In this paper we prove that there are infinite families of primitive groups contained in the union of imprimitive groups  and  propose a new hierarchy for primitive groups    based on that fact.  In addition to the previous results and hierarchy, we introduce some algorithms to handle permutations, provide the corresponding GAP implementation,  solve some open problems, and propose a large list of open problems.}
\end{abstract}
\maketitle

\baselineskip = 1.3\baselineskip

\section{Introduction}
In many practical situations we know that a primitive group contains a given permutation and we want to know which group it can be; in some other practical situations we know the group and would like to know if it contains a permutation of some given type.  For example, a group $G\le S_n$ is said to be  non-synchronizing  if it is contained in the automorphism group of a non-trivial primitive graph with complete core, that is, with clique number equal to chromatic number (see for example, \cite{abcrs,ac,ac2,ArCaSt15,arnold_steinberg,ck,neumann:sectionregular}).
 When trying to check if some group is synchronizing, typically, we have only  partial information about the graph but enough to say that it has an automorphism of some type, and the goal would be to have in hand a classification of the primitive groups containing permutations of that type. As an illustration of this, the key ingredient in some of the results in \cite{ac} was the observation that the primitive graph under study has a $2$-cycle automorphism and hence the automorphism group of the graph is the symmetric group. For many more examples of the importance of knowing the groups that contain permutations of a given type, please see Praeger's slides \cite{slides}.

This type of investigation is certainly very natural since it appears on the eve of group theory, with Jordan, Burnside, Marggraff, but  the difficulty of the problem is well illustrated by the very slow progress throughout  the twentieth century. Given the new tools available (chiefly the classification of finite simple groups), the topic seems to have  new momentum (see, for example, \cite{guest,jones,king,liebeck,lopes,muller}).

Let $S_n$ denote the symmetric group on $n$ points; a permutation $g\in S_n$ is said to be \emph{imprimitive} if there exists an imprimitive group containing $g$.  An imprimitive  group is said to be \emph{minimally imprimitive}  if it contains no transitive proper subgroup \cite{hulpke}. An imprimitive group $G\le S_n$ is said to be \emph{maximally imprimitive} if for all $g\in S_n\setminus G$, the group $\langle g,G\rangle$ is primitive. The next result, whose proof is straightforward, provides some alternative characterizations of imprimitive permutations.

\begin{thrm}\label{2.1}
Let $n$ be a natural number and let $g\in S_n$. Then the following are equivalent:
\begin{enumerate}
\item\label{2} there exists an imprimitive group $H\le S_n$ such that $g\in H$;
\item\label{2a} there exists a maximally imprimitive group $H\le S_n$ such that $g\in H$;
\item\label{2b} there exists an imprimitive group $H\le S_n$ such that $\langle g , H\rangle$ is imprimitive;
\item\label{3} there exists a minimally imprimitive group $H\le S_n$ such that $\langle g, H\rangle$ is imprimitive;
\item\label{4} there exists a permutation $h\in S_n$ such that $h$ and $g$ are conjugate (under $S_n$) and $h\in H$, for some imprimitive group $H\le S_n$.
\end{enumerate}
\end{thrm}

{\color{black}We use the idea of imprimitive permutations to propose a new  hierarchy for primitive groups. Roughly speaking this hierarchy measures the size of sets of imprimitive permutations contained in a given primitive group.

In Section \ref{prelim} we introduce some definitions,  basic results {\color{black} and GAP algorithms.  In Section \ref{infinite} we prove that there are infinite families of primitive groups entirely composed of imprimitive permutations.}   In Section \ref{hier} we introduce the hierarchy and prove some results about it. In Section \ref{problopes} we solve some of the problems posed in \cite{lopes}. The paper ends with a list of open problems.}

\section{Preliminaries }\label{prelim}

Let $S_n$ denote the symmetric group on $n$ points. A permutation $g\in S_n$ is said to be \emph{primitive} if it fails to be imprimitive. It follows that a permutation $g\in S_n$  is primitive if and only if any transitive group $G\le S_n$ containing $g$ is primitive; the permutation $g$ is said to be \emph{strongly primitive} if the only transitive groups containing $g$ are the symmetric and alternating groups.

Asymptotically, almost all permutations are strongly primitive. This follows from the theorem of {\L}uczak and Pyber~\cite{lp} {\color{black}(the asymptotics are given in \cite{dfg,efg}):}

\begin{thrm}
The proportion of strongly primitive permutations in $S_n$ tends to $1$ as
$n\to\infty$.
\end{thrm}

In this connection, note that in \cite{dfg} a good upper bound is given for the
number of permutations which are primitive but not strongly primitive. Indeed,
according to \cite{cnt}, for almost all $n$ (a set of density $1$), the only
primitive groups of degree $n$ are symmetric and alternating groups, and so
there is no difference between primitive and strongly primitive permutations.

The next theorem provides a more technical, but much more practical, characterization of imprimitive permutations.  We start with some definitions.
{\color{black}
\begin{defin} Let $k$ be a positive integer, and $(m_1,\ldots,m_l)$ a partition
of $m$. Then the partition $(km_1,\ldots,km_l)$ {\color{black} of $km$} is said to be an
\emph{ic-partition} (for ``imprimitive cycle'') of type $(k,m)$.
\end{defin}

Note that an ic-partition of type $(k,m)$ is also an ic-partition of type
$(k/d,md)$ for any divisor $d$ of $k$. For example, the partition  $(30,24,12)$ of $66$ is an ic-partition of type $(2,33)$ since $(30,24,12)=(2\times 15,2\times 12, 2\times 6)$ and $15+12+6=33$; it is also an ic-partition of type $(6,11)$ since  $(30,24,12)=(6\times 5,6\times 4, 6\times 2)$ and $5+4+2=11$.

\begin{defin}
For a partition $P$  with $r$ parts,
a \emph{clustering} of $P$ is a partition of the set of parts of $P$
into parts (called \emph{clusters}) $P_1,\ldots,P_r$ (each of which is a
partition).
\end{defin}

For example, $((2,2),(2,1,1),(1,1))$ is a clustering of $(2,2,2,1,1,1,1)$. The
clusters are partitions of $4$, $4$ and $2$ respectively.

\begin{defin}\label{ipart}
An \emph{i-partition} (for ``imprimitive'') of type $(k,m)$ is a partition
of $km$ which has a clustering into clusters which are
ic-partitions of $k_im$ of type $(k_i,m)$ for $i=1,\ldots,r$, where $(k_1,\ldots,k_r)$ is a partition of $k$.
\end{defin}
}
\paragraph{Example} The partition $(1,5,10,10,10,10,10,10)$ is an i-partition of type $(11,6)$. {\color{black}This is shown by the clustering $((1,5), (10,10,10,10,10,10))$.} For the partition $(1,5)=(1\times 1,1\times 5)$ is an ic-partition of type $(1,6)$, and $(10,\ldots,10)=({10\times 1,\ldots,10\times 1})$ is an ic-partition of type $(10,6)$.

\begin{thrm}\label{ipartition}
The permutation $g$ is contained in an imprimitive permutation group with $k$ blocks of size $m$ if and only if the cycle partition of $g$ on $\{1,\ldots,km\}$ (with $k,m>1$)
is an i-partition of type $(k,m)$.
\end{thrm}

\begin{proof} First we observe that the cycle partition of $g$ is an
ic-partition of type $(k,m)$ if and only if $g$ is contained in an imprimitive
group as in the theorem and induces a cyclic permutation on the set of blocks.
For if $g$ permutes the blocks cyclically, then we return to the same block
after $k$ steps, so every cycle has length divisible by $k$; and conversely,
if every cycle has length divisible by $k$, we obtain the block system by
assigning the points in each cycle to the $k$ blocks in cyclic order.

Now for an arbitrary permutation, if its cycle type is an i-partition, then we can construct a block system on the union of the cycles in each ic-partition, and this block system is preserved by $g$. The converse is clear.
\end{proof}

As a simple illustration, we show:

\begin{thrm}
Let $\Omega$ be a set of size $n$ and let $P=(p_1,\ldots,p_l)$ be a partition of $n$, in which the parts  have a common divisor larger than $1$. Then any permutation of cycle-type $P$ is imprimitive.
\end{thrm}

For if the greatest common divisor is $m$, and $mk=n$, then take the partition $(p_1/m,\ldots,p_l/m)$ of $k$, and for each $i$ take the ic-partition of $p_i$ with a single part, to verify that $P$ is an i-partition.

The following variant of Theorem~\ref{ipartition} turns out to be useful for
testing.
\begin{lem}\label{ipartest}
{\color{black}If a} partition $P$ of $n$ points is an $i$-partition of type $(n/m,m)${\color{black}, then} there exists a clustering
$P_1,\ldots, P_r$ of $P$ with $l_i=\sum_{p\in P_i} p$ and
$g= { \gcd(\ell_1,\ldots,\ell_r)}$ such that $m$ is a divisor of $g$ and $P_i$
is an ic-partition of type $(k_i=l_i/g,g)$.  Conversely, if such a clustering
exists, then $P$ is an $i$-partition of type {\color{black}$(n/m,m)$}.

In particular, $P$ is the cycle shape of an imprimitive permutation if and
only if such a clustering exists for which $1<g<n$.
\end{lem}
\begin{proof}
Note that $n=\sum l_i$ is a multiple of $g$, and hence any divisor of $g$  is a
divisor of $n$.

Suppose that $P_1,\ldots,P_r$ is a clustering that identifies $P$ as an
i-partition with $P_i$ an ic-partition of type $(k_i,m)$. Then $m$ divides
every $l_i$ and thus divides $g$. Furthermore, $P_i$ is also an ic-partition
of type $(k_im/g,g)$, as $k_im/g$ divides $k_i$. The converse statement is
the definition of an i-partition.

The last statement follows immediately from Theorem~\ref{ipartition}.
\end{proof}

A final result concerns using a pair of permutations to guarantee primitivity.
We follow the terminology of Definition~\ref{ipart}.

\begin{defin}
The i-type of a partition is the set of pairs $(k,m)$ for which the partition has an i-partition of type $(k,m)$. The i-type of a permutation is the i-type of its cycle partition.
\end{defin}

A permutation is primitive if and only if its i-type is empty. Now suppose that a transitive group $G$ contains two imprimitive permutations $g_1$ and $g_2$ whose i-types are disjoint. Then necessarily $G$ is primitive.

For example, $G=2^4.3^2:4$, the $8$th primitive group of degree $16$ in GAP, contains an
element $g_1$ with cycle structure $[8,8]$, and an element $g_2$ with cycle
structure $[3,3,3,3,3,1]$. Then $g_1$ is of $i$-types $(2,8)$ and $(8,2)$, while
$g_2$ is of $i$-type $(4,4)$ only.

\subsection{An algorithm to identify primitive permutations}\label{algorithm}

By Theorem~\ref{ipartition}, the test for a permutation to be primitive can be
based purely on  {its} cycle structure; it needs to be an
$i$-partition. In particular it is sufficient to test conjugacy class
representatives if testing whether a group contains a primitive permutation.

We therefore describe an algorithm that tests whether a given partition
$P$ of a composite number $n$, consisting of $l$ parts, is an i-partition.
We will write $P=(p_1,\ldots,p_l)$ as a collection of parts, assuming
without loss of generality that $p_i\ge p_{i+1}$.
By Lemma~\ref{ipartest} we need to consider all clusterings of $P$.

The starting point for this is a process for enumerating partitions of a set
(the set being the parts of $P$). Following~\cite[Section~7.2.1.5]{knuth4A},
partitions of a set of cardinality $l$ correspond to restricted growth
strings $(a_i)_{i=1}^l$ of length $l$ (the $i$-th entry of the string gives
the part number in which the $i$-th set element is placed), that is,
$a_{j+1}\le 1+\max(a_1,\ldots,a_j)$.

Since the partition $P$ might have parts of equal size, this parameterization
of set partitions will produce partitionings that are equal, i.e. if the
partition is $P=(A,a,B,b)$ (with upper/lower case to distinguish equal
entries: $A=a$, $B=b$) then $(A,B)(a,b)$ and $(A,b),(a,B)$ are, nominally
different, equal partitionings. That is, if $I$ is a set of indices (because
$P$ is ordered it will be in fact an interval) such
that $p_i$ is constant for $i\in I$, and $\pi$ is a permutation with support
$I$, the sequences $(a_i)$ and $(a_{i^\pi})$ result in equivalent
partitionings.

We incorporate this equivalency in the construction process by requiring
that if $P$ is constant on the interval $I$, the sequence $(a_i)$ is
non-decreasing on $I$. In the construction algorithm for the restricted
growth strings, Algorithm~H of~\cite[Section~7.2.1.5]{knuth4A}, this
condition can only be violated if $a_j\leftarrow 0$ in step H6. We modify
this step by setting $k\leftarrow j$ at the start of this step and
replacing, as long as $p_j=p_k$,
the assignment $a_j\leftarrow 0$ with $a_j\leftarrow a_k$.

Each string then defines a clustering $P_1,\ldots,P_r$ of $P$. By~\ref{ipartest}
we calculate $g=\gcd(l_i)_{i=1}^r$ with $l_i=\sum_{p\in P_i} p$.
If $1<g<n$ test whether for each $i$ all parts of $P_i$ are of length a
multiple of $l_i/g$. If so, $P$ is an i-partition.

The cost of this algorithm grows rapidly with the number of parts of the
partition $P$. Such partitions tend to have many parts
of small size. To speed up handling in these cases, we begin by testing
whether for any proper divisor $m$ of $n$ we can cluster the parts of $P$
into subsets of cardinality $m$ each (in this case each cluster is an
ic-partitions of type $(1,m)$ and thus $P$ will be an i-partition).
We do so with a greedy algorithm that increases each partial cluster
by adding the largest remaining part that does not push the cluster content
over $m$. This test is very quick and succeeds for example if there are many
 of size $1$. Only if the greedy algorithm fails we start the full
search for clusterings.
\smallskip

We have implemented this test in {\sf GAP}~\cite{GAP4}, the code is
available at
\url{http://www.math.colostate.edu/~hulpke/examples/primitivepermutation.g}.

\section{{\color{black}Primitive groups without primitive permutations}}\label{infinite}

The aim of this section is to provide examples of primitive groups fully composed by imprimitive permutations. We start with three infinite families.

\medskip

\paragraph{\textbf{First construction:} Covering with imprimitive subgroups}

{This} construction provides a very flexible way of producing primitive groups without primitive permutations.
Let $H$ and $K$ be finite permutation groups and consider the wreath product $G = H \wr K$ in the product action. By \cite[Lemma 2.7.A.]{dixon} this is primitive if and only if
\begin{enumerate}
 \item  $H$ is primitive and not cyclic of prime order;
 \item  $K$ is transitive.
 \end{enumerate}

The goal now is to produce a primitive group which is a union of imprimitive subgroups. We can do this from the wreath product construction if $K$ is a transitive group which is the union of intransitive subgroups (this simply means that $K$ contains no cyclic transitive subgroup -- so there are many examples, e.g. non-cyclic groups acting regularly).
Take, for example, $H=S_3$, and $K=(C_2)^m$ acting regularly. Then $G$ is primitive of degree $3^{2^m}$ and the minimum number of generators for $G$ is at least $m$ since
if we take fewer than $m$ elements of $G$, then the subgroup they generate projects onto a proper (and hence intransitive) subgroup of $K$, and so is imprimitive; thus we require at least $m$ elements to generate a primitive subgroup.

\medskip

\paragraph{\textbf{Second construction:} Some affine groups.}

\medskip

As some primitive affine groups can also be written as a product action, namely primitive subgroups of $G\wr S_n$ where $G\le\AGL(1,p)$ is transitive but not cyclic, we focus on such groups.  In light of the previous construction, we mainly focus on $G\wr C_n$ with the product action where $C_n$ is the cyclic group of order $n$.  We begin by fixing some notation and a general result which characterizes imprimitive elements of $\AGL(k,p)$.

Of course, $\AGL(k,p)$ contains a normal regular elementary abelian subgroup, which we call $E$, and $\AGL(k,p) = \GL(k,p)\cdot E$.  The following result characterizes imprimitive elements of $\AGL(k,p)$.

\begin{lem}\label{characteristic polynomial}
Let $g\in\AGL(k,p)$ with $g = Ae$, where $A\in\GL(k,p)$ and $e\in E$.  The element $g$ is imprimitive if and only if the characteristic polynomial of $A$ is reducible.
\end{lem}

\begin{proof}

Since the group $E$ acts regularly, the $E$-invariant partitions are the coset
partitions corresponding to the subgroups of $E$ (this well-known result is a special case of \cite[Theorem 7.5]{Wielandt1964}). {\color{black}So, if $A$ has an invariant subspace, then $A$ and hence also $g=Ae$ preserves the partition into cosets of this subspace.}

Now $A$ has an
invariant subspace (subgroup of $E$) if and only if its characteristic
polynomial is irreducible. (One way round is clear: if $A$ has an invariant
subspace then it is represented by a matrix of the form
$\begin{pmatrix}B&0\\C&D\end{pmatrix}$. Conversely, suppose that the characteristic
polynomial $\phi(x)$ of $A$ factorises as $f(x)g(x)$. If $f(A)=0$, then the
$A$-submodule spanned by a vector $v$ has dimension at most the degree of $f$.
Similarly if $g(A)=0$. Otherwise, $Ef(A)$ is annihilated by $g(A)$ so is a
proper submodule of $E$.)
\end{proof}

Now let $S\in\GL(k,p)$ be the permutation matrix which shifts the coordinates of ${\mathbb F}_p^k$ one to the left (so that $S$ is a cyclic group of order $k$ acting regularly on the coordinates of ${\mathbb F}_p^k$), and $D\in\GL(k,p)$ be a diagonal matrix.  Let $d_i$ be the entry in $D$ in row $i$ and column $i$, and $d = \prod_{i=1}^kd_i$.  Straightforward computations will show that $DS$ has characteristic polynomial $x^k - d$.  Also, any element of $\GL(k,p)$ contained in $\AGL(1,p)\wr C_k$ will {\color{black}either fix a subspace of dimension dividing $k$ (if it projects onto an element of order smaller than $k$ in $C_k$) or} have the form $DS$, with $D$ and $S$ as above (this is easiest to see computing in $\AGL(1,p)\wr C_k$).  We give our first example of a subgroup of $\AGL(k,p)\cap(\AGL(1,p)\wr C_k)$ that is imprimitive.

\begin{lem}
Let $p$ be prime and $G = D_p\wr C_4\le\AGL(4,p)$.  Then every element of $G$ is imprimitive.
\end{lem}

\begin{proof}
By definition and comments above, {\color{black} if $g\in G$ is }contained in $\GL(4,p)$ {then $g$ either \color{black} preserves a subspace of dimension $1$ or $2$, or it has} characteristic polynomial $x^4 \pm 1$, both of which are reducible. (Since $p^2-1$ is divisible by $8$, the roots of this polynomial lie in $\mathbb{F}_{p^2}$.) The result follows by Lemma \ref{characteristic polynomial}.
\end{proof}

\begin{thrm}\label{bigonedirection}
Let $k$ be an integer and $p$ be prime, $R = \{s\in{\mathbb F}_p^*:s = t^k{\rm\ for\ some\ }t\in{\mathbb F}_p^*\}$, $G\le\AGL(1,p)\wr C_k{\le\AGL(k,p)}$ be transitive, and $g\in G$.  Write $g = D_gS^{a_g}$ where $D_g\in\GL(k,p)$ is diagonal and $a_g$ is an integer.  {\color{black}Suppose that,} whenever $g\in G\cap\GL(k,p)$ cyclically permutes the coordinates of ${\mathbb F}^k$ as a $k$-cycle then $\det(D_g)\in R$. {\color{black}Then every element of $G$ is imprimitive.}
\end{thrm}

\begin{proof}
Suppose that whenever $g\in G\cap\GL(k,p)$ cyclically permutes the coordinates of ${\mathbb F}^k$ as a $k$-cycle then $\det(D)\in R$.  By comments above, the characteristic polynomial of $g$ is $x^k - d$ where $d = \det(D)$ which is reducible as $d$ is a $k$-th root.  The result follows by Lemma \ref{characteristic polynomial}.
\end{proof}

Applying the covering argument from the previous construction and observing that any two regular cyclic subgroups of $S_k$ are conjugate in $S_k$, we have the following result.

\begin{cor}\label{bigonedirection corollary}
Let $p$ be prime and $k$ a positive integer, and $R = \{s\in{\mathbb F}_p^*:s = t^k{\rm\ for\ some\ }t\in{\mathbb F}_p^*\}$,
and $G = \la x\mapsto rx {+} b:{r\in R}, b\in\Z_p\ra\le\AGL(1,p)$.  Then every element of $G\wr S_k$ with the product action is imprimitive.
\end{cor}

Note that in the previous result if $\gcd(k,p-1) = 1$, then $G = \AGL(1,p)$.

\begin{cor}
Let 
$p\ge 5$, and $H\le \AGL(1,p)$ consist of all elements of the form $x\mapsto ax + b$, where $a$ is a quadratic residue modulo $p$.  Then $H\wr S_2$ is primitive but contains no primitive elements and contains a normal  imprimitive subgroup of index $2$.
\end{cor}

\begin{proof}
We first apply Corollary \ref{bigonedirection corollary} with $k = 2$, in which case $R$ is the set of quadratic residues modulo $p$.  As $p\ge 5$, $R\not=\{1\}$ and so $H = \la x\mapsto rx {+} b:{r\in R}, b\in\Z_p\ra\not \cong\Z_p$.  Then every element of $H\wr S_2$ is imprimitive by Corollary \ref{bigonedirection corollary} and this group is primitive of degree $p^2$ by \cite[Lemma 2.7.A.]{dixon}.  Finally, $H\times H\tl H\wr S_2$ is a normal imprimitive subgroup of index $2$.
\end{proof}

The next result will allow us to verify the unsurprising fact that there are primitive permutation groups that are product actions which have primitive elements.

\begin{thrm}\label{AGL imprimitive converse}
Let $q = 2$ or $3$, $p$ be prime, and $R = \{s\in{\mathbb F}_p^*:s = t^q{\rm\ for\ some\ }t\in{\mathbb F}_p^*\}$, $G\le\AGL(1,p)\wr C_q$, and $g\in G$.  Write $g = D_gS^{a_g}$ where $D_g\in\GL(q,p)$ is diagonal and $a_g$ is an integer.  Then every element of $G$ is imprimitive if and only if whenever $g\in G\cap\GL(q,p)$ cyclically permutes the coordinates of ${\mathbb F}^q$ as a $q$-cycle then $\det(D_g)\in R$.
\end{thrm}

\begin{proof}
In view of Theorem \ref{bigonedirection}, we need only to prove the converse.  Suppose that every element of $G$ is imprimitive.  We will show that that if $\det(D_g) = d\not\in R$ then $g$ is primitive by Theorem  \ref{ipartition}.  So suppose that $d\not\in R$.  It is clear that $g$ fixes the $0$ vector.  If $g$ fixes any other vector ${\bf v}$, then ${\bf v}$ is an eigenvector of $g$ with eigenvalue $1$.  As the characteristic polynomial {of $g$} is $x^q - d$ {and $d\not\in R$}, this implies that $d = 1\in R$.  Straightforward computations will show that $g^q = dI_n$.  Now, if $s\in\Z_p^*$ but there is no $t\in\Z_p^*$ with $t^q = s$, then $q\vert(p-1)$ and the highest power $q^a$ of $q$ dividing $p - 1$ also divides the order of $s$.  We conclude that any orbit of $g$ that is not the singleton orbit $\{(0,0)\}$ has order a multiple $q^{a+1}$ as $g^q$ has order a multiple of $q^a$.  Now, if $g$ is imprimitive and preserves the invariant partition ${\cal B}$, then $g$ fixes the block $B\in{\cal B}$ that contains $(0,0)$, so $B-\{(0,0)\}$ is a union of orbits of $g$.  Suppose that $\vert B\vert = p$.  Then $p - 1$ is a sum of multiples of $q^{a+1}$, and so the highest power of $q$ dividing $p - 1$ is $q^{a + 1}$, a contradiction.  Suppose that $\vert B\vert = p^2$ (and so $q = 3$).  Then $p^2 - 1$ is a multiple of $q^{a + 1}$ and as the highest power of $q$ that divides $p - 1$ is $q^a$, we conclude that $q$ divides $p + 1$.  However, $\gcd(p-1,p+1) = 2\not = q$, a contradiction which establishes the result.
\end{proof}

\paragraph{\textbf{Third construction:} From a problem of Wielandt}

Another infinite class of examples comes from the following observation. {\color{black}We say that the \emph{spectrum} of a subgroup of $S_n$ is the set of cycle types of elements it contains.}

\begin{thrm}
{\color{black}Two faithful actions of a group $G$ with the same permutation character have the same spectrum. Hence if}
the group $G$ has two faithful transitive actions, one primitive and one imprimitive, with the same permutation character, then every element in the primitive action of $G$ is an imprimitive permutation.
\end{thrm}

\begin{proof} If $g$ has $c_i$ cycles of length $i$, then for any divisor $d$ of the order of $g$, the number of fixed points of $g^d$ is $\sum_{f\mid d}c_f$; then M\"obius inversion shows that the numbers $c_i$ are determined by the character values on powers of $g$. {\color{black}So the first statement is true; and the second follows immediately.}
\end{proof}

Such groups are not easy to find: Wielandt asked in 1979 whether they could exist. An infinite family based on the exceptional groups of type $E_8$ was found by Guralnick and Saxl~\cite{gs}, while a sporadic example in the group $J_4$ was found by Breuer~\cite{breuer}.

It is worth pointing out that this construction gives examples of primitive groups where all the permutations have a common $i$-type{\color{black}, namely the block size and number of blocks of its imprimitive companion}.

\paragraph{\textbf{Fourth construction:} Some groups of diagonal type.}

\begin{thrm}
Let $T$ be a nonabelian simple group.  Then no element of the primitive group $T^2$ with the diagonal action is primitive.
\end{thrm}

\begin{proof}
That $T^2$ with the diagonal action is primitive follows from \cite[Theorem 4.5a (i)]{dixon}.  Also, $T^2$ contains $T$ as a regular subgroup, and we lose no generality by assuming that $T_L\le T^2$, where $T_L$ is the left regular representation of $T$.  In $T^2$, $T_L$ is centralized by subgroup isomorphic to $T$, and the centralizer in $S_T$ is $T_R$ by \cite[Lemma 4.2A]{dixon}.  So we may assume without loss of generality that $T\times T = T_L\times T_R$.  Now, both $T_L$ and $T_R$ are regular, and being so are imprimitive by \cite[Theorem 1.5A]{dixon} as they are simple and nonabelian and so have proper nontrivial subgroups.  Hence any element of $T^2$ of the form $(t,1)$ or $(1,t)$ with $t\in T$ is imprimitive.  Consider an element of the form $(t,s)$, where $t,s\in T$ and neither is the identity.  Let $H = \la(r,1),(t,s):r\in T\ra$, and let $\pi_i:T^2\rightarrow T$ be projection in the coordinate $i$, $i = 1,2$.  Of course, $\pi_1(H) = T$, while $\pi_2(H) = \la s\ra < T$.  Additionally, $H = T_L\times \la s\ra$, and is transitive as $T_L$ is transitive.  However, $\la(1,s)\ra\tl H$ is not transitive, and so its orbits form a complete block system of $H$.  Thus $H$ is imprimitive and so $(r,s)$ is imprimitive.  Thus every element of $T^2$ is imprimitive as required.
\end{proof}

\medskip

\paragraph{\textbf{Miscellaneous examples}}

We know finitely many further examples of such groups. Some come from the
computer search with the GAP implementation of the algorithm above.
We investigated the primitive groups of degree up to 120.

\begin{longtable}{r|r|l|r|r|l|r|r|l|r|r|l}
\hline
$n$
&\#
&$G$
&$n$
&\#
&$G$
&$n$
&\#
&$G$
&$n$
&\#
&$G$
\\
\hline
$15$
&$2$
&$A_{6}$
&$64$
&$2$
&$2^{6}{:}D_{14}$
&$81$
&$40$
&$3^{4}{:}2^{2+2+2}$
&$100$
&$10$
&$A_{10}^{2}{.}2^{2}$
\\
$ $
&$3$
&$S_{6}$
&$ $
&$6$
&$2^{6}{:}3^{2}{:}3$
&$ $
&$42$
&$3^{4}{:}D_{16}{:}4$
&$ $
&$11$
&$A_{10}^{2}{.}4$
\\
$16$
&$4$
&$(A_{4}\times A_{4}){:}2$
&$ $
&$7$
&$2^{6}{:}7{:}6$
&$ $
&$43$
&$3^{4}{:}(SA_{16}{:}2){:}2$
&$ $
&$12$
&$S_{10}\wr S_{2}$
\\
$ $
&$7$
&$2^{4}{.}S_{3}\times S_{3}$
&$ $
&$9$
&$2^{6}{:}3^{2}{:}S_{3}$
&$ $
&$49$
&$3^{4}{:}2^{3}{:}A_{4}$
&$ $
&$13$
&$A_{6}\wr S_{2}$
\\
$ $
&$8$
&$2^{4}{.}3^{2}{:}4$
&$ $
&$10$
&$2^{6}{:}3^{2}{:}S_{3}$
&$ $
&$50$
&$3^{4}{:}(Q_{8}{:}2){:}S_{3}$
&$ $
&$14$
&$A_{6}^{2}{.}2^{2}$
\\
$ $
&$10$
&$(S_{4}\times S_{4}){:}2$
&$ $
&$16$
&$2^{6}{:}(7\times D_{14})$
&$ $
&$51$
&$3^{4}{:}Q_{8}{.}S_{3}{:}2$
&$ $
&$15$
&$A_{6}^{2}{.}2^{2}$
\\
$25$
&$4$
&$5^{2}{:}Q(8)$
&$ $
&$17$
&$2^{6}{:}3^{2}{:}D_{12}$
&$ $
&$52$
&$3^{4}{:}(2\times Q_{8}){:}6$
&$ $
&$16$
&$A_{6}^{2}{.}4$
\\
$ $
&$5$
&$5^{2}{:}D(2\cdot 4)$
&$ $
&$18$
&$2^{6}{:}(3^{2}{:}3){:}4$
&$ $
&$53$
&$3^{4}{:}Q_{8}{.}S_{3}{:}2$
&$ $
&$17$
&$A_{6}^{2}{.}2^{2}$
\\
$ $
&$11$
&$5^{2}{:}D(2\cdot 4){:}2$
&$ $
&$25$
&$2^{6}{:}(3^{2}{:}3){:}Q_{8}$
&$ $
&$54$
&$3^{4}{:}(SA_{16}{:}2){:}3$
&$ $
&$18$
&$A_{6}^{2}{.}4$
\\
$27$
&$1$
&$3^{3}{.}A_{4}$
&$ $
&$26$
&$2^{6}{:}(3^{2}{:}3){:}8$
&$ $
&$55$
&$3^{4}{:}(Q_{8}{:}3){:}2^{2}$
&$ $
&$19$
&$A_{6}^{2}{.}2^{2}$
\\
$ $
&$3$
&$3^{3}(A_{4}\times 2)$
&$ $
&$27$
&$2^{6}{:}(3^{2}{:}3){:}D_{8}$
&$ $
&$56$
&$3^{4}{:}Q_{8}{.}S_{3}{:}2$
&$ $
&$20$
&$A_{6}^{2}{.}D_{8}$
\\
$ $
&$4$
&$3^{3}{.}2{.}A_{4}$
&$ $
&$28$
&$2^{6}{:}7{:}7{:}6$
&$ $
&$57$
&$3^{4}{:}(Q_{8}{:}3){:}4$
&$ $
&$21$
&$A_{6}^{2}{.}2^{3}$
\\
$ $
&$5$
&$3^{3}{.}S_{4}$
&$ $
&$29$
&$2^{6}{:}7^{2}{:}S_{3}$
&$ $
&$58$
&$3^{4}{:}(\mbox{GL}_{1}(3)\wr D_{4})$
&$ $
&$22$
&$A_{6}^{2}{.}D_{8}$
\\
$ $
&$8$
&$3^{3}(S_{4}\times 2)$
&$ $
&$33$
&$2^{6}{:}(3^{2}{:}3){:}SD_{16}$
&$ $
&$59$
&$3^{4}{:}Q_{16}{:}D_{8}$
&$ $
&$23$
&$A_{6}^{2}{.}D_{8}$
\\
$28$
&$1$
&$\mbox{PGL}_{2}(7)$
&$ $
&$38$
&$2^{6}{:}7^{2}{:}(3\times S_{3})$
&$ $
&$60$
&$3^{4}{:}(4\times 8){:}4$
&$ $
&$24$
&$A_{6}^{2}{.}D_{8}$
\\
$ $
&$4$
&$\mbox{PSU}_{3}(3)$
&$ $
&$42$
&$2^{6}{:}(\mbox{GL}_{3}(2)\wr 2)$
&$ $
&$61$
&$3^{4}{:}2^{2+3+1+1}$
&$ $
&$25$
&$A_{6}^{2}{.}D_{8}$
\\
$ $
&$5$
&$\mbox{P$\Gamma$U}(3, 3)$
&$ $
&$47$
&$2^{6}{:}3{.}S_{6}$
&$ $
&$63$
&$3^{4}{:}2^{2+3+1+1}$
&$ $
&$26$
&$A_{6}^{2}{.}(2\times 4)$
\\
$ $
&$9$
&$\mbox{PSL}_{2}(27)$
&$ $
&$48$
&$2^{6}{:}3{.}A_{6}$
&$ $
&$64$
&$3^{4}{:}2^{3+4}{:}4$
&$ $
&$27$
&$A_{6}^{2}{.}D_{8}$
\\
$ $
&$10$
&$\mbox{PGL}_{2}(27)$
&$ $
&$58$
&$2^{6}{:}S_{8}$
&$ $
&$66$
&$3^{4}{:}2^{2+3+1+1}$
&$ $
&$28$
&$A_{6}^{2}{.}(2\times 4)$
\\
$35$
&$1$
&$A_{8}$
&$ $
&$59$
&$2^{6}{:}A_{8}$
&$ $
&$72$
&$3^{4}{:}(\mbox{GL}_{1}(3)\wr A_{4})$
&$ $
&$29$
&$A_{6}^{2}{.}(2\times 4)$
\\
$ $
&$2$
&$S_{8}$
&$ $
&$60$
&$2^{6}{:}S_{7}$
&$ $
&$73$
&$3^{4}{:}Q_{8}{:}S_{4}$
&$ $
&$30$
&$A_{6}^{2}{.}2^{2}{:}4$
\\
$ $
&$3$
&$A_{7}$
&$ $
&$61$
&$2^{6}{:}A_{7}$
&$ $
&$74$
&$3^{4}{:}2^{3}{:}S_{4}$
&$ $
&$31$
&$A_{6}^{2}{.}(2\times D_{8})$
\\
$ $
&$4$
&$S_{7}$
&$ $
&$62$
&$2^{6}{:}\Sigma U(3, 3)$
&$ $
&$75$
&$3^{4}{:}(2\times Q_{8}){:}A_{4}$
&$ $
&$32$
&$A_{6}^{2}{.}2^{2}{:}4$
\\
$36$
&$3$
&$M_{10}$
&$ $
&$63$
&$2^{6}{:}SU_{3}(3)$
&$ $
&$76$
&$3^{4}{:}2^{3+2}{:}S_{3}$
&$ $
&$33$
&$A_{6}^{2}{.}(2\times D_{8})$
\\
$ $
&$4$
&$\mbox{PGL}_{2}(9)$
&$ $
&$64$
&$2^{6}{:}\mbox{PGL}_{2}(7)$
&$ $
&$77$
&$3^{4}{:}(Q_{8}{.}S_{3}{:}2){:}2$
&$ $
&$34$
&$A_{6}^{2}{.}(2\times D_{8})$
\\
$ $
&$5$
&$\mbox{P$\Gamma$L}(2, 9)$
&$ $
&$65$
&$A_{8}\wr S_{2}$
&$ $
&$78$
&$3^{4}{:}(SA_{16}{:}2){:}6$
&$ $
&$35$
&$A_{6}^{2}{.}2^{2}{:}4$
\\
$ $
&$8$
&$\mbox{\mbox{PSp}}_{4}(3)$
&$ $
&$66$
&$A_{8}^{2}{.}2^{2}$
&$ $
&$79$
&$3^{4}{:}(SA_{16}{:}2){:}6$
&$ $
&$36$
&$\mbox{P$\Gamma$L}(2, 9)\wr S_{2}$
\\
$ $
&$9$
&$\mbox{\mbox{PSp}}_{4}(3){:}2$
&$ $
&$69$
&$\mbox{PSL}_{2}(7)\wr S_{2}$
&$ $
&$80$
&$3^{4}{:}Q_{8}{.}S_{3}{:}4$
&$102$
&$1$
&$\mbox{PSL}_{2}(17)$
\\
$ $
&$13$
&$(A_{6}\times A_{6}){:}2$
&$ $
&$70$
&$\mbox{PSL}_{2}(7)^{2}{.}2^{2}$
&$ $
&$81$
&$3^{4}{:}Q_{8}{.}S_{3}{:}2^{2}$
&$105$
&$1$
&$\mbox{PSL}_{3}(4){.}2$
\\
$ $
&$14$
&$(A_{6}\times A_{6}){:}2^{2}$
&$ $
&$71$
&$\mbox{PSL}_{2}(7)^{2}{.}4$
&$ $
&$82$
&$3^{4}{:}2^{2+3+1+2}$
&$ $
&$2$
&$\mbox{PSL}_{3}(4){.}2$
\\
$ $
&$15$
&$(A_{6}\times A_{6}){:}4$
&$ $
&$72$
&$\mbox{PGL}_{2}(7)\wr S_{2}$
&$ $
&$83$
&$3^{4}{:}(\mbox{GL}_{1}(3)\wr D_{4}){:}2$
&$ $
&$3$
&$\mbox{PSL}_{3}(4){.}2^{2}$
\\
$ $
&$16$
&$(S_{6}\times S_{6}){:}2$
&$65$
&$1$
&$\mbox{PSL}_{2}(5^{2})$
&$ $
&$85$
&$3^{4}{:}8^{2}{:}4$
&$ $
&$4$
&$\mbox{PSL}_{3}(4){.}S_{3}$
\\
$ $
&$17$
&$(A_{5}\times A_{5}){:}2$
&$ $
&$2$
&$\mbox{P$\Sigma$L}(2, 5^{2})$
&$ $
&$88$
&$3^{4}{:}\mbox{SL}_{2}(3){:}A_{4}$
&$ $
&$5$
&$\mbox{PSL}_{3}(4){.}6$
\\
$ $
&$18$
&$(A_{5}\times A_{5}){.}4$
&$ $
&$3$
&$\mbox{PSU}_{3}(4)$
&$ $
&$91$
&$3^{4}{:}(\mbox{GL}_{1}(3)\wr S_{4})$
&$ $
&$6$
&$\mbox{PSL}_{3}(4){.}D_{12}$
\\
$ $
&$19$
&$((A_{5}\times A_{5}){:}2)2$
&$66$
&$1$
&$\mbox{PGL}_{2}(11)$
&$ $
&$92$
&$3^{4}{:}Q_{8}^{2}{:}6$
&$ $
&$7$
&$S_{8}$
\\
$ $
&$20$
&$(S_{5}\times S_{5}){:}2$
&$ $
&$2$
&$M_{11}$
&$ $
&$93$
&$3^{4}{:}Q_{8}^{2}{:}S_{3}$
&$112$
&$1$
&$\mbox{PSU}_{4}(3)$
\\
$40$
&$1$
&$\mbox{\mbox{PSp}}_{4}(3)$
&$ $
&$3$
&$M_{12}$
&$ $
&$94$
&$3^{4}{:}\mbox{GL}_{2}(3){:}D_{8}$
&$ $
&$2$
&$\mbox{PSU}_{4}(3){.}2$
\\
$ $
&$2$
&$\mbox{\mbox{PSp}}_{4}(3){:}2$
&$77$
&$1$
&$M_{22}$
&$ $
&$96$
&$3^{4}{:}\mbox{SL}_{2}(3){:}S_{4}$
&$ $
&$3$
&$\mbox{PSU}_{4}(3){.}2$
\\
$ $
&$3$
&$\mbox{\mbox{PSp}}_{4}(3)$
&$ $
&$2$
&$M_{22}{.}2$
&$ $
&$97$
&$3^{4}{:}(2^{3}{:}A_{4}){:}S_{3}$
&$ $
&$4$
&$\mbox{PSU}_{4}(3){.}2$
\\
$ $
&$4$
&$\mbox{\mbox{PSp}}_{4}(3){:}2$
&$81$
&$4$
&$3^{4}{:}D_{16}$
&$ $
&$98$
&$3^{4}{:}\mbox{GL}_{2}(3){:}(3\times S_{3})$
&$ $
&$5$
&$\mbox{PSU}_{4}(3){.}2^{2}$
\\
$45$
&$1$
&$\mbox{PGL}_{2}(9)$
&$ $
&$5$
&$3^{4}{:}SA_{16}$
&$ $
&$100$
&$3^{4}{:}Q_{8}^{2}{:}D_{12}$
&$ $
&$6$
&$\mbox{PSU}_{4}(3){.}4$
\\
$ $
&$2$
&$M_{10}$
&$ $
&$6$
&$3^{4}{:}Q_{8}{:}2$
&$ $
&$101$
&$3^{4}{:}(\mbox{SL}_{2}(3)\wr 2)$
&$ $
&$7$
&$\mbox{PSU}_{4}(3){.}2^{2}$
\\
$ $
&$3$
&$\mbox{P$\Gamma$L}(2, 9)$
&$ $
&$8$
&$3^{4}{:}SD_{16}$
&$ $
&$102$
&$3^{4}{:}\mbox{GL}_{2}(3){:}S_{4}$
&$ $
&$8$
&$\mbox{PSU}_{4}(3){.}D_{8}$
\\
$ $
&$4$
&$\mbox{\mbox{PSp}}_{4}(3)$
&$ $
&$13$
&$3^{4}{:}SA_{16}{:}2$
&$ $
&$103$
&$3^{4}{:}(2^{3}{:}A_{4}){:}S_{3}$
&$117$
&$1$
&$\mbox{PSL}_{3}(3){.}2$
\\
$ $
&$5$
&$\mbox{\mbox{PSp}}_{4}(3){:}2$
&$ $
&$14$
&$3^{4}{:}2^{2}{:}4{:}2$
&$ $
&$104$
&$3^{4}{:}(2^{3}{:}2^{2}){:}(3^{2}{:}4)$
&$120$
&$1$
&$S_{7}$
\\
$ $
&$6$
&$A_{10}$
&$ $
&$15$
&$3^{4}{:}SA_{16}{:}2$
&$ $
&$106$
&$3^{4}{:}Q_{8}^{2}{:}S_{3}^{2}$
&$ $
&$2$
&$A_{9}$
\\
$ $
&$7$
&$S_{10}$
&$ $
&$16$
&$3^{4}{:}2^{3}{:}2^{2}$
&$ $
&$107$
&$3^{4}{:}(2^{3}{:}2^{2}){:}3^{2}{:}D_{8}$
&$ $
&$6$
&$\mbox{PSL}_{3}(4)$
\\
$49$
&$2$
&$7^{2}{:}S_{3}$
&$ $
&$17$
&$3^{4}{:}D_{16}{:}2$
&$84$
&$1$
&$A_{9}$
&$ $
&$7$
&$\mbox{PSL}_{3}(4){.}2$
\\
$ $
&$7$
&$7^{2}{:}D(2\cdot 6)$
&$ $
&$18$
&$3^{4}{:}2^{2+2+1}$
&$ $
&$2$
&$S_{9}$
&$ $
&$8$
&$\mbox{PSL}_{3}(4){.}2$
\\
$ $
&$12$
&$7^{2}{:}3\times D(2\cdot 3)$
&$ $
&$19$
&$3^{4}{:}Q_{16}{:}2$
&$85$
&$1$
&$\mbox{\mbox{PSp}}_{4}(4)$
&$ $
&$9$
&$\mbox{PSL}_{3}(4){.}2$
\\
$ $
&$21$
&$7^{2}{:}3\times D(2\cdot 6)$
&$ $
&$20$
&$3^{4}{:}2^{2+1+2}$
&$ $
&$2$
&$\mbox{\mbox{PSp}}_{4}(4){.}2$
&$ $
&$10$
&$\mbox{PSL}_{3}(4){.}2^{2}$
\\
$52$
&$1$
&$\mbox{PSL}_{3}(3){.}2$
&$ $
&$21$
&$3^{4}{:}D_{16}{:}2$
&$91$
&$1$
&$\mbox{PSL}_{2}(13)$
&$ $
&$11$
&$S_{8}$
\\
$60$
&$1$
&$A_{5}^{2}$
&$ $
&$22$
&$3^{4}{:}(2\times Q_{8}){:}2$
&$ $
&$2$
&$\mbox{PGL}_{2}(13)$
&$ $
&$14$
&$\mbox{\mbox{PSp}}_{6}(2)$
\\
$ $
&$2$
&$A_{5}^{2}{.}2$
&$ $
&$23$
&$3^{4}{:}SA_{16}{:}2$
&$ $
&$3$
&$\mbox{PSL}_{2}(13)$
&$ $
&$16$
&$O^+(8, 2)$
\\
$ $
&$3$
&$A_{5}\wr S_{2}$
&$ $
&$30$
&$3^{4}{:}Q_{8}{.}S_{3}$
&$ $
&$4$
&$\mbox{PGL}_{2}(13)$
&$ $
&$17$
&$\mbox{PSO}^+(8, 2)$
\\
$ $
&$4$
&$A_{5}\wr S_{2}$
&$ $
&$31$
&$3^{4}{:}(Q_{8}{:}3){:}2$
&$100$
&$1$
&$J_{2}$
&$ $
&$18$
&$A_{10}$
\\
$ $
&$5$
&$A_{5}^{2}{.}2^{2}$
&$ $
&$32$
&$3^{4}{:}(\mbox{GL}_{1}(3)\wr 4)$
&$ $
&$2$
&$J_{2}{.}2$
&$ $
&$19$
&$S_{10}$
\\
$63$
&$1$
&$\mbox{PSU}_{3}(3)$
&$ $
&$33$
&$3^{4}{:}4^{2}{:}4$
&$ $
&$5$
&$A_{5}\wr S_{2}$
&$ $
&$20$
&$A_{16}$
\\
$ $
&$2$
&$\mbox{PSU}_{3}(3){.}2$
&$ $
&$34$
&$3^{4}{:}Q_{8}{:}D_{8}$
&$ $
&$6$
&$A_{5}^{2}{.}2^{2}$
&$ $
&$21$
&$S_{16}$
\\
$ $
&$3$
&$\mbox{PSU}_{3}(3)$
&$ $
&$35$
&$3^{4}{:}2^{3}{:}D_{8}$
&$ $
&$7$
&$A_{5}^{2}{.}4$
&&&\\
$ $
&$4$
&$\mbox{PSU}_{3}(3){.}2$
&$ $
&$37$
&$3^{4}{:}2^{2+3+1}$
&$ $
&$8$
&$S_{5}\wr S_{2}$
&&&\\
$ $
&$5$
&$\mbox{\mbox{PSp}}_{6}(2)$
&$ $
&$39$
&$3^{4}{:}D_{16}{:}4$
&$ $
&$9$
&$A_{10}\wr S_{2}$
&&&\\
\hline
\end{longtable}

\bigskip

These groups belong to the following classes (in GAP's version of the  O'Nan-Scott Theorem):

\begin{description}
\item[``1''] Affine.
\item[``2''] Almost simple.
\item[``3a''] Diagonal, Socle consists of two normal subgroups.
\item[``3b''] Diagonal, Socle is minimal normal.
\item[``4c''] Product action with the first factor primitive of type 2.
\end{description}
\medskip

The next table contains the synchronizing groups without primitive permutations.
\[
\begin{array}{l|lllll}
\hline
\mbox{Degree}&\mbox{Groups}\\
\hline
28&\PSU(3, 3)&\PGammaU(3, 3)&\PSL(2, 27)&\PGL(2, 27)\\
36&\PSp(4, 3)
&\PSp(4, 3):2\\
40&\PSp(4, 3)&\PSp(4, 3):2&\PSp(4, 3)&\PSp(4, 3):2\\
63&\PSU(3, 3)&\PSU(3, 3).2&\PSp(6, 2)\\
64&2^6:(3^2:3):Q_8&2^6:(3^2:3):8
&2^6:(3^2:3):\SD_{16}&2^6:\SigmaU(3, 3)&2^6:\SU(3, 3)\\
65&\PSL(2, 5^2)&\PSigmaL(2, 5^2)
&\PSU(3, 4)\\ \hline
\end{array}
\]
They are very few; this means that up to degree $100$, the overwhelming majority  of primitive groups without primitive permutations are non-synchronizing.

The next observation may provide some more sporadic examples of primitive groups without primitive permutations. Let $G$ be a primitive group with an imprimitive normal subgroup $N$ of index $2$. Then all elements of $N$ are imprimitive, and it is only necessary to check half the elements of $G$. For example, let $G = \PGL(2,11)$, with degree $66$. The elements of $G$ not in the imprimitive normal subgroup have cycle types $(6,12,12,12,12,12)$, $(1,5,10,10,10,10,10,10)$, and two types corresponding to cubes of the first and fifth powers of the second. Theorem \ref{ipartition} shows that all these permutations are imprimitive.

\section{An infinite hierarchy for primitive groups}\label{hier}

The previous results suggest the following hierarchy for primitive groups based on the cycle-type of permutations they contain.

A set $S\subseteq G\le S_n$ is said to be independent if
$$(\forall g\in S)\ g\not\in \langle S\setminus\{g\}\rangle.$$
Let $k$ and $n$ be two natural numbers. The class ${\AP} k$ contains all degree $n$ primitive groups $G$ in which all independent $k$-subsets are primitive, that is, given any independent $k$-set $S\subseteq G$, and given any minimally imprimitive   group $H\le S_n$, we have that $\langle S,H\rangle$ is a primitive group.
For example, it is obvious that every transitive group of prime degree
belongs to ${\AP} 1$.

The class ${\EP} k$ contains all degree $n$ primitive groups $G$ in which there exist some  $k$-subsets which are primitive.
For example,  $A_6$ (degree $15$) does not contain primitive permutations,
but contains several primitive pairs and hence this group belongs to
${\EP} 2$.

The class $\mathbf{NP} k$ contains all degree $n$ primitive groups $G$ in which there are no $k$-subsets which are primitive.
For example,  $A_6$ (degree $15$)  belongs to  $\mathbf{NP} 1$. Also, the groups $S_3\wr 2^m$ belongs to $\mathbf{NP}m$. So our hierarchy is infinite.

It is clear that if $G$ belongs to ${\AP} k$, then it also belongs to
${\AP} m$, for all $m\ge k$.

Also, if $G$ belongs to $\mathbf{NP} k$, then it also belongs
to $\mathbf{NP} m$, for all $m\le k$. Similarly, if $G$ belongs to
$\mathbf{EP} k$, then it belongs to $\mathbf{EP} m$, for all $m\ge k$.

A group $G$ belongs to the class ${\NAP} k$ if the group belongs to
$\mathbf{AP} k$ and to $\mathbf{NP} (k-1)$. Similarly, a group belongs to
${\NEP} k$ if the group belongs to $\mathbf{EP} k$ and to $\mathbf{NP} (k-1)$.

For a group $G$ we denote by $\NEP(G)$ the value of $k$ such that $G$
belongs to ${\NEP} k$; $\NAP(G)$ is defined in the same way.

\subsection{Testing for $\EP$ and $\AP$}

To test membership of a group $G$ in $\EP$ or $\AP$ for a given $k$,
we need to iterate through the $k$-subsets of $G$. Clearly it is sufficient
to do so up to conjugacy in $N=N_{S_n}(G)$.

An algorithm that will run through sequences of group elements up to conjugacy
is given in \cite[Section V.5]{hulpkediss} in the context of searching for
homomorphisms. Its input consists of $N$-conjugacy classes of elements of
$G$.

While enumerating sequences clearly enumerates sets,
there is a substantial amount of duplication.
By demanding that the conjugacy class for the $i+1$-th sequence element is not smaller
(in some arbitrary ordering of conjugacy classes)
than the class for the $i$-th element, it is
possible to eliminate a substantial amount of duplication of the same set by
different sequences. We found this to be a reasonable tradeoff of runtime
efficiency and implementation cost.

We consider first the
case of $\EP$, that is we test whether a primitive set exists:
For every sequence $A$ we test whether $\gen{A}$ is a
primitive subgroup of $S_n$ (in which case the property is true), or whether
$\gen{A}$ is transitive (in which case $A$ can be discarded). We collect a
pool of those
$A$ which generate intransitive subgroups.

We reduce this pool by eliminating sequences if the groups generated by them
are $N$-conjugate.

Next we test for a failure of $A$ to be primitive within $G$: We determine
the maximal transitive but imprimitive subgroups of $G$ and (if they are not
too large) calculate the
subgroups thereof. If $A$ is conjugate to any of these subgroups it cannot be
primitive.

It turns out that for all examples we considered so far this strategy was
successful in either eliminating all $A$'s or finding one $A$ that actually
generated a primitive group. It is possible however that neither case holds,
that is an $A$ that generates an intransitive group cannot be eliminated.

In this case we observe that the definition of a primitive set is clearly
equivalent to the property that there exists a transitive, imprimitive, group
$H\le S_n$ such that $A\subset H$ (because then $\gen{H,A}=H$). We thus test
whether $A$ is contained in a maximal imprimitive subgroup $H\le S_n$. These
groups
are wreath products of symmetric groups, parameterized by partitions of $n$
into blocks of equal size, and thus can be constructed easily.

For efficiency, we note that we need to consider $H$ only up to conjugacy by
$M=N_{S_n}(\gen{A})$. This is achieved by selecting $H$ up to conjugacy (note
that $H$ is maximal in $S_n$ and thus self-normalizing) and let
$r$ run through a set of representatives of the double cosets $M\setminus
S_n/H$. For each $r$ we test whether $A^r\subset H$.
\medskip

The test for $\AP$ proceeds similar, but with different stopping criteria: We
iterate over sequences $A$. If $\gen{A}$ is primitive, or $A$ is not
independent we discard $A$. If $\gen{A}$ is transitive and imprimitive the
group is not in $\AP$. We collect those $A$ for which $\gen{A}$ is
intransitive.

As soon as a single independent imprimitive sequence is found the search terminates
(as the group is not in $\AP k$).

Otherwise we proceed as for $\EP$ and reduce sequences up to subgroup
conjugacy and test that indeed all remaining sequences are primitive.

\subsection{Building groups belonging to $\mathbf{NEP} k$ and $\mathbf{NAP} k$}

According to GAP, up to degree $63$, every primitive group in $\mathbf{NP} 1$ belongs to $\mathbf{NEP} 2$.

\subsection{A partial order and an equivalence relation on primitive groups}\label{same spectrum}

{\color{black}Recall that, for} a permutation group, $G\le S_n$, the set of cycle-types of elements of
$G$ is said to be the \emph{spectrum} of $G$.
The relation \emph{equality of spectrum} is obviously an equivalence
relation on the subgroups of $S_n$, with conjugacy implying equality of
spectrum. These equivalence classes can be ordered by inclusion of spectra,
yielding a lattice that is an image of the lattice of classes of groups
under inclusion.

We note that the spectrum cannot determine primitivity. {\color{black}The solutions to
Wielandt's problem in Section~\ref{infinite} demonstrate this. Also,}
primitive group
$4$ in degree $25$ -- $5^2{:}Q(8)$ -- has three maximal subgroups of index
$2$ which all are imprimitive and all have the same spectrum as the whole group.

{\color{black} The table below shows that affine groups with degrees that are perfect square predominate for small degrees, but the examples answering Wielandt's question  show that there are others:}
\[
\begin{array}{|l|l|}
\hline
\mbox{Degree}&\mbox{Groups with the same spectrum}\\
\hline

9  &\bigl\{3^2{:}Q_8, 3^2{:}4\bigr\}\\
16 & \bigl\{ 2^4.S_3 \times S_3, (A_4 \times A_4){:}2 \bigr\}\\
25 & \bigl\{5^2{:}D(2*4){:}2, 5^2{:}D(2*4)\bigr\}, \bigl\{5^2{:}((Q_8{:}3)'2), 5^2{:}4\times D(2*3)\bigr\}\\
   & \bigl\{5^2{:}(Q_8{:}3), 5^2{:}Q_{12}\bigr\}\\
49 & \bigl\{7^2{:}3 \times Q_{12}, 7^2{:}3 \times (Q_8{:}3) \bigr\}, \bigl\{ 7^2{:}3 \times Q_8, 7^2{:}12 \bigr\}, \\
   &\bigl\{ 7^2{:}(3 \times Q_{16}), 7^2{:}24 \bigr\}, \bigl\{7^2{:}Q_{12}, 7^2{:}Q_8{:}3 \bigr\}, \bigl\{ 7^2{:}Q_8, 7^2{:}4 \bigr\}, \bigl\{ 7^2{:}Q_{16}, 7^2{:}8 \bigr\}\\
64& \bigl\{ 2^6{:}(3^2{:}3){:}4, 2^6{:}(3^2{:}3){:}Q_8\bigr\},
\bigl\{2^{6}{:}(3\wr A_{3}){:}2, 2^{6}{:}(3\wr A_{3}){:}2 \bigr\}\\
81& \bigl\{
3^{4}{:}D_{16}, 3^{4}{:}SA_{16}, 3^{4}{:}SA_{16}{:}2,
  3^{4}{:}D_{16}{:}2, 3^{4}{:}2^{2+2+1}, 3^{4}{:}Q_{16}{:}2,
    3^{4}{:}2^{2+1+2}, 3^{4}{:}D_{16}{:}2, 3^{4}{:}SA_{16}{:}2, \\
& 3^{4}{:}2^{2+2+2}, 3^{4}{:}D_{16}{:}4 \bigr\},
 \bigl\{  3^{4}{:}Q_{8}{:}2, 3^{4}{:}2^{3}{:}2^{2},
  3^{4}{:}(2\times Q_{8}){:}2 \bigr\},\\
&\bigl\{ 3^{4}{:}SA_{16}{:}2, 3^{4}{:}Q_{8}{:}D_{8}, 3^{4}{:}2^{2+3+1},
  3^{4}{:}D_{16}{:}4, 3^{4}{:}(SA_{16}{:}2){:}2,
    3^{4}{:}Q_{16}{:}D_{8}, 3^{4}{:}2^{2+3+1+1},
    3^{4}{:}2^{2+3+1+1}\bigr\},\\
&\bigl\{ 3^{4}{:}(Q_{8}{:}3){:}2, 3^{4}{:}(Q_{8}{:}3){:}2^{2} \bigr\},
\bigl\{ 3^{4}{:}8{.}D_{8}, 3^{4}{:}8{.}D_{8}{:}2\bigr\},
\bigl\{ 3^{4}{:}D_{16}{:}4, 3^{4}{:}16{:}4 \bigr\},\\
& \bigl\{ 3^{4}{:}(8\times D_{10}), 3^{4}{:}5{:}2^{2+1+2} \bigr\},
\bigl\{  3^{4}{:}(Q_{8}{:}2){:}S_{3}, 3^{4}{:}Q_{8}{.}S_{3}{:}2,
  3^{4}{:}Q_{8}{.}S_{3}{:}2, 3^{4}{:}Q_{8}{.}S_{3}{:}2^{2} \bigr\},\\
&\bigl\{ 3^{4}{:}(SA_{16}{:}2){:}3, 3^{4}{:}(SA_{16}{:}2){:}6 \bigr\},
\bigl\{ 3^{4}{:}(Q_{8}{:}3){:}4, 3^{4}{:}2^{3+2}{:}S_{3} \bigr\},
\bigl\{  3^{4}{:}2^{2+3+1+1}, 3^{4}{:}2^{2+3+1+2} \bigr\},\\
&\bigl\{ 3^{4}{:}(Q_{8}{.}S_{3}{:}2){:}2, 3^{4}{:}Q_{8}^{2}{:}6\bigr\},
\bigl\{ 3^{4}{:}4{.}A_{6}, 3^{4}{:}4{.}S_{5} \bigr\},
\bigl\{ 3^{4}{:}4{.}A_{6}, 3^{4}{:}4{.}S_{5} \bigr\}
\\

\hline
\end{array}
\]

There is at least one easy way to guarantee two primitive groups of the same degree that are not isomorphic have the same spectrum.

\begin{lem}\label{samespectrum}
Let $K$ be a primitive group of degree $m$ acting on $\Omega$ that is not a regular cyclic group of prime order, and $G$ and $H$ nonisomorphic transitive subgroups of $S_n$ with the same spectrum.  Then $K\wr G$ and $K\wr H$ with the product action are nonisomorphic primitive groups with the same spectrum.
\end{lem}

\begin{proof}
That $K\wr G$ and $K\wr H$ are primitive follows from \cite[Lemma 2.7.A.]{dixon} as $K$ is cyclic of prime degree.  Also, $K\wr G$ and $K\wr H$ are not isomorphic as $G$ and $H$ are not isomorphic.  Clearly $K\wr G$ and $K\wr H$ are contained in $K\wr S_n$.  It is also clear that as $K^n\le (K\wr G)\cap(K\wr H)$, if $g\in K\wr G$ has cycle structure different from every element of $K\wr H$, it must be the case that $g$ nontrivially permutes the coordinates of $\Omega^n$.  For $g\in K\wr S_n$, we denote the action of $g$ on the coordinates of $\Omega^n$ by $\bar{g}$.  Now, as $G$ and $H$ have the same spectrum, there exists $\bar{h}\in H$ with the same cycle structure as $\bar{g}$.  Then $\bar{h}$ and $\bar{g}$ are conjugate in $S_n$, and so there exists $\delta\in K\wr S_n$ such that $\overline{\delta^{-1}g\delta} = \bar{h}$.  But then $\delta^{-1}g\delta\in K\wr H$ and there is no element of $K\wr G$ of cycle structure different from that of $g$.  Thus $K\wr G$ and $K\wr H$ have the same spectrum.
\end{proof}

\begin{thrm}
There exist primitive groups $G$ and $H$ of the same degree with the same spectrum and $G$ is in $NEPk$ but $H$ is not.
\end{thrm}

\begin{proof}
Let $K\le S_n$ be primitive but not cyclic of prime order generated by two elements (for example, a subgroup of $\AGL(1,q)$ for some prime $q$ that is not cyclic).  There exists a nonabelian group $G_1$ of order $p^5$ for $p\ge 5$ a prime such that every element of $G_1$ has order $p$ and is generated by two elements \cite{Bender1927}.  Let $H_1 = \Z_p^5$.  Then the spectrum of $G_1$ and $H_1$ are the same, and so by Lemma \ref{samespectrum} $G = K\wr G_1$ and $H = K\wr H_1$ have the same spectrum and are not isomorphic.  As $H_1$ is generated by five elements, we have that at least five elements are needed to generate $H$.  So $H\in NEP4$.  Let $g_1$ and $g_2$ generate $G_1$ and $g_3$ and $g_4$ be elements of $K^{p^5}$ that generate $K$ in the first coordinate and are the identity in all others.  Then $G = \la g_1,g_2,g_3,g_4\ra$ so that $G\not\in NEP4$.
\end{proof}

There are infinite families of nonisomorphic primitive groups with the same spectrum.

\begin{lem}
Let $p\equiv 3\ (\mod 4)$.  Then $p^2:Q_8$ and $p^2:4$ are nonisomorphic primitive groups of degree $p^2$ with the same spectrum.
\end{lem}

\begin{proof}
First recall that $\SL(2,p)$ has order $(p^2 - 1)p$.  As $p\equiv 3\ (\mod 4)$, a Sylow $2$-subgroup of $\SL(2,p)$ has order $8$.  By \cite[Theorem 8.3 (ii)]{Gorenstein1968} a Sylow $2$-subgroup of $\SL(2,p)$ is generalized quaternion, and a Sylow $2$-subgroup of $\SL(2,p)$ is $Q_8$.  Let $I$ be the $2\times 2$ identity matrix.  Then $-I\in\SL(2,p)$ and so is contained in a Sylow $2$-subgroup $Q$ of $\SL(2,p)$.  As $Q\cong Q_8$ has a unique subgroup of order $2$, namely $\la -I\ra$ whose only fixed point in its action on ${\mathbb F}_{p^{\color{black}2}}$ is $(0,0)$, we see that each element of order $4$ in $Q$ is a product of disjoint $4$-cycles with one fixed point.  Let $g\in Q$ be a $4$-cycle, and $P = \{(i,j)\mapsto (i + a,j + b):a,b\in\Z_p\}$ so that $P\tl \ASL(2,p)$.  Then $P\tl \la P,g\ra {\color{black}= p^2:4}$ is solvable, and by \cite[Theorem 4 (4)]{DobsonW2002} and Burnside's Theorem \cite[Theorem 3.5B]{dixon} a maximal solvable imprimitive subgroup of $S_{p^2}$ {\color{black} with Sylow $p$-subgroup $\Z_p^2$} is $\AGL(1,p)\times\AGL(1,p)$.  Again as $p\equiv 3\ (\mod 4)$ the group $\AGL(1,p)\times\AGL(1,p)$ has Sylow $2$-subgroup $\Z_2^2$, and so ${\color{black}p^2:4}$ is primitive.  Then each primitive subgroup of $\ASL(2,p)$ of the form $p^2:4$ is a subgroup of index $2$ in some subgroup of $\ASL(2,p)$ of the form $p^2:Q_8$.  Given that each element of order $4$ in $p^2:Q_8$ is a product of $4$-cycles and one fixed point, it is easy to see that $p^2:4$ has the same spectrum as $p^2:Q_8$.
\end{proof}

\section{On the problems in \cite{lopes}}\label{problopes}

 Lopes \cite{lopes} introduces  two families of primitive permutations, proves results about them and asks a number of questions. In this section we recall the concepts, extend his results and solve some of the questions.

 \begin{defin}
Let  $n$ be a natural number and consider a partition $P = (p_1, p_2, \ldots, p_l)$ of $n$, for $l>1$. If the $p_i$ can be rearranged into subsets so that the sum of the elements of each subset is equal to a constant $m$,

that divides $n$, then $P$ is said to be an $m$-partition.

If the largest part, say $p_l$,  is divisible by $m$, and the remaining parts $p_1,\ldots,p_{l-1}$ can be rearranged in sets whose sum is equal to $m$, then we say that $P$ is a special $m$-partition. (Note that a special $m$-partition is not generally an $m$-partition!)
\end{defin}

These concepts can be illustrated by examples that we borrow from \cite{lopes}. Consider $P=(2,3,5)$; this is an $m$-partition  as the parts can be arranged into two sets $\{\{2,3\},\{5\}\}$, with the sum of the elements in each set being equal to $5$. An example of a special $m$-partition, is $P=(2,3,10)$. The largest part ($10$) has $5$ among its divisors, and the remaining parts can be rearranged in a set $\{2,3\}$ whose sum is $5$. Another example is $P=(1,2,5,7,17,19,23,111)$. The largest part ($111$) has $37$ among its divisors, and the other parts can be rearranged in sets adding up to that number: $\{\{2,5,7,23\},\{1,17,19 \}\}$.

We state (and use Theorem \ref{ipartition} to prove) the main results of \cite{lopes}.
\begin{thrm}(\cite{lopes})\label{lps}
Let $\Omega$ be a set of size $n$ and let $P=(p_1,\ldots,p_l)$ be a partition of $n$, in which the $p_i$ are pairwise distinct.  Let $S_n$ be the symmetric group on $\Omega$. Then
 \begin{enumerate}
\item\label{(1)} if $l=2$ and $p_1,p_2$ are co-prime, then any permutation $g\in S_n$ of type $P$ is primitive; if in addition $p_1,p_2>1$, then $g$ is strongly primitive;
\item\label{(2)} if $l\ge 3$, the $p_i$ are pairwise  co-prime, and $P$ is neither a{n} $m$-partition nor a special $m$-partition, then any permutation in $S_n$ of cycle-type $P$ is strongly primitive.
\end{enumerate}
\end{thrm}
\begin{proof}
It is clear that a $2$-part partition has non-empty i-type if and only if the
two parts are not co-prime, giving half of (\ref{(1)}). Let $G$ be a primitive group containing a permutation of type $P=(p_1,p_2)$, with $p_1,p_2>1$. Since they are co-primes, it follows that $G$ contains a $p_1$-cycle and a $p_2$-cycle. If one of them is a $2$-cycle, then $G$ is the symmetric group \cite{ac}. Otherwise, $G$ contains a cycle of length smaller than $n-2$, which by \cite{jones} implies that $G$ contains the alternating group. The proof of (\ref{(1)}) is complete.

Regarding (\ref{(2)}), observe that as above a primitive group containing a permutation with such
a cycle-type must contain a cycle of length at most $n-2$,  thus containing the alternating group. It remains to prove that a transitive group  containing a permutation with such a cycle-type is primitive.  Suppose that a partition with its parts pairwise co-prime is an i-partition
of type $(m,k)$. Let $(k_1,k_2,\ldots)$ be the corresponding partition of $k$.
Then there are parts of size divisible by $k_i$ summing to $mk_i$ for each $i$.
The pairwise co-prime condition shows that, for each $i$, either $k_i=1$ (so
there are some parts summing to $m$), or there is a single part of size
$mk_i$. The second alternative can hold at most once (else there are two
parts with a common factor $m$); if it holds once, then we have a special
$m$-partition, otherwise an $m$-partition. This gives (\ref{(2)}) and also
shows how the two types of partitions emerge naturally.
\end{proof}

{\color{black} The next result extends to full classifications the partial results in \cite{lopes}. To do this we only need to use \cite{jones}.}
\begin{thrm}(\cite{jones,lopes})
Let $\Omega$ be a set of size $n$ and let $P\in \{(1,n-1),(n)\}$ be a partition of $n$, in which the parts are pairwise distinct.  
Now we have the following:
 \begin{enumerate}
\item\label{(3)} if $P=(1,n-1)$, then any permutation of this cycle-type is primitive, and a proper primitive group $G$ contains a permutation of such type  if and only if one of the following holds:
\begin{enumerate}
\item $\AGL(d,q)\le G\le \AGAMMAL(d,q)$, with $n=q^d$ and $d\ge 1$, for some prime power $q$;
\item $G={\color{black} \PSL(2,p)}$ or $\PGL(2,p)$ with $n=p+1$ for some prime $p\ge 5$;
\item $G=\Mat_{11},\Mat_{12}$ or $\Mat_{24}$ with $n=12,12$ or $24$ respectively;
\end{enumerate}
\item if $P=(n)$, then
 \begin{enumerate}
\item if $n$ is prime, then every permutation in $S_n$ is primitive, and the proper primitive groups containing a permutation of cycle-type $P$ are the following:
\begin{enumerate}
\item $C_p\le G\le \AGL(1,p)$,  with $n=p$ prime;
\item $\PGL(d,q)\le G \le \PGAMMAL (d,q)$ with $n=(q^d-1)/(q-1)$ and $d\ge 2$ for some prime power $q$, and $n$ prime;
\item $G={\color{black}\PSL(2,11)}, \Mat_{11}$ or $\Mat_{23}$ with $n=11,11$ or $23$ respectively;
\end{enumerate}
\item if $n$ is not prime, then every $n$-cycle is an imprimitive permutation as the cyclic group it generates is imprimitive.
\end{enumerate}
\end{enumerate}
\end{thrm}

{\color{black} In the end of his paper,  Lopes \cite{lopes} asks the following questions whose answers easily follow from our results:}

\begin{enumerate}
\item Is it always possible to identify the primitive permutations of a primitive group?
\item\label{al2} Are there primitive permutations of cycle-type a relatively prime [special] $m$-partition?
\item Are there primitive permutations whose cycle-types have parts with multiplicity greater than one, but the different parts are still mutually prime?
\end{enumerate}

Regarding the last question, the partitions $P=(1,1,n-2)$, for $n>2$, satisfy the property and about them we have the following result.

\begin{thrm}
A permutation of type $(1,1,n-2)$ is primitive if and only if $n$ is odd. In addition a degree $n$ primitive group  $G$ contains a permutation of type $(1,1,n-2)$ if and only if  $\PGL(2,q)\le G \le \PGAMMAL (2,q)$, with $n=q+1$ for some prime power $q$.
\end{thrm}

\begin{proof} The first part of the theorem is a simple consequence of Theorem~\ref{ipartition}.  The second part of the result  follows from \cite{jones}.
\end{proof}

Theorem \ref{ipartition} and Subsection \ref{algorithm} answer the first question above.

Finally, to solve question (\ref{al2}), suppose we have a $m$-partition of $n$, and suppose we can organize the parts in sets that  sum up to $k$ each. Then $k$ must divide $n$ (as the total sum is $n$) and a permutation of this cycle-type  is found in the $k$-fold direct power of $S_{n/k}$. But this group lies in the imprimitive wreath product $S_{n/k}\wr S_k$. So the permutation cannot be primitive.

Alternatively, it can be shown that both $m$-partitions and special $m$-partitions are i-partitions, and therefore imprimitive. Let $p$ be an $m$-partition with cycle-type $(a_1, a_2, \cdots, a_n)$. Now partition this partition into parts where the sum of the elements is $m$. All those parts are ic-partitions of type $(1,m)$, and therefore this new partition is a clustering. For special $m$-partitions, the argument is similar, except that we have a cluster where the larger part is alone.

As an example, consider the partitions given after the definition of $m$-partition. We have the partitions $P_1=(2,3,5)$ and $P_2=(2,3,10)$. For $P_1$, we have the clustering $\{(2,3), (5)\}$, where the clusters are both $(1,5)$ ic-partitions. For $P_2$, we have the clustering $\{(2,3), (10)\}$, where the clusters are, respectively, a $(1,5)$ ic-partition and a $(2,5)$ ic-partition.

\section{Problems}

\begin{prob}
Does Theorem \ref{AGL imprimitive converse} hold for all integers $q\ge 2$?
\end{prob}

\begin{prob}
Are there primitive groups $G,H\le S_n$ with the same spectrum, but such that $G$ belongs to $\mathbf{NAP} k$, but $H$ does not?
\end{prob}

\begin{prob}
Is there a natural transversal for the spectrum equivalence classes for the degree $n$ primitive groups?
\end{prob}

\begin{prob}
Is the class  $\mathbf{NEP} k$ [$\mathbf{NAP} k$] non-empty, for all natural $k$?
\end{prob}

We observed that $A_6$ acting on pairs is a primitive group containing only imprimitive permutations; and also observed that since $A_6$ is $2$-generated, there exists a primitive group containing no primitive permutation, but containing primitive pairs of {\color{black} imprimitive permutations}. Therefore, the following theorem might be of some use to attack the previous problem.

\begin{thrm}(\cite{Lu2})	Let
 $G\le S_n$ be a primitive permutation group. Then
the smallest number of elements needed to generate $G$ is at most
 \[
\frac{C\log n}{\sqrt{\log \log n}}, 	
 \]where $C$ is a universal constant.
\end{thrm}

Let $G\le S_n$ and denote by $d(G)$ the smallest number of elements of $G$ needed to generate this group. Let $m(n)$ be the maximum of the set $\{d(G)\mid G\le S_n \mbox{ is primitive}\}$. To handle the previous question we certainly will need a very good lower bound for $m(n)$.  So we propose the {following} problem.

\begin{prob}
Let $n$ be a natural number. Find good lower bounds for $m(n)$.
\end{prob}

 By Corollary 2.4 of \cite{orbits}, every $2$-homogeneous group $G$ has $d(G)=2$ and hence to compute $m(n)$ we only have to consider the other primitive groups. We observe that the bound $\frac{cn}{\sqrt{\log n}}$ is the best possible apart from constants; examples are in \cite{konew}. The best possible constants have been found by G. Tracey in \cite{tracey}. This observation does not solve the problem (since when $n$ is prime then $m(n)=2$), but it shows that the upper bound cannot be improved in general. (See also \cite{tracey2}.)

\begin{prob}  	
Is it true that $\mathbf{NP} 1$ intersects every type of primitive groups in GAP's version of the O'Nan-Scott theorem?

Answer similar questions for $\mathbf{NAP} k$ and in $\mathbf{NEP} k$.
\end{prob}

\begin{prob}
Classify the primitive groups in  $\mathbf{NAP} k$ and in $\mathbf{NEP} k$.
\end{prob}

\begin{prob}
Can the examples included in the table{\color{black}s} of subsection \ref{infinite} {\color{black} and subsection \ref{same spectrum}} be extended to infinite families?
\end{prob}

We saw above that the class of $\mathbf{NP} 1$ is to a large extent contained in the class of non-synchronizing groups. Therefore the next question looks natural.
\begin{prob}  Is there any natural $m$ such that for every $k>m$ every group in 	
$\mathbf{NAP} k$ [$\mathbf{NEP} k$] is non-synchronizing?
\end{prob}

{\color{black} The classification of synchronizing groups is still to be done. Here we propose an (hopefully) easier problem.
\begin{prob}	
Classify the synchronzing groups in which every element is an imprimitive permutation.
\end{prob}
}

\section{Acknowledgments}

The first author acknowledges that this work was developed within FCT projects CAUL (PEst-OE/MAT/UI0143/2014)  and CEMAT-CI\^{E}NCIAS (UID/Multi/04621/2013).

The second author acknowledges that this work was developed within Funda\c{c}\~{a}o Caluste Gulbenkian "Novos Talentos em Matem\'{a}tica 2015" grant.

The work of the fifth author was supported by
Simons Foundation Collaboration Grant~244502.

The last author acknowledges support from FCT (Funda\c c\~ao para a Ci\^encia e a Tecnologia), Portugal, through project FCT EXCL/MAT-GEO/0222/2012, ``Geometry and Mathematical Physics''.

\providecommand{\bysame}{\leavevmode\hbox to3em{\hrulefill}\thinspace}
\providecommand{\MR}{\relax\ifhmode\unskip\space\fi MR }
\providecommand{\MRhref}[2]{%
  \href{http://www.ams.org/mathscinet-getitem?mr=#1}{#2}
}
\providecommand{\href}[2]{#2}

\end{document}